\documentclass[11pt]{amsart}
\usepackage{amssymb}

\usepackage{palatino}
\input amssym.def
\usepackage{amsmath, amsfonts}
\usepackage{amssymb}
\usepackage{amscd, mathtools}
\usepackage[mathscr]{eucal}
\usepackage{palatino}
\newfont{\cyrr}{wncyr10}

\newcommand{\Z}{{\mathbb Z}}

\newcommand{\C}{{\mathbb C}}

\newcommand{\SL}{{\rm SL}}

\renewcommand{\mod}{{\, \rm mod \, }}

\newtheorem{thm}{Theorem}

\newtheorem{lem}[thm]{Lemma}

\newtheorem{prop}[thm]{Proposition}

\begin{document}

\title[Fourier-Coefficients]{Large Fourier coefficients of 
half-integer weight modular forms} 

\author{S. Gun, W. Kohnen and K. Soundararajan}

\address[S. Gun]   
{Institute of Mathematical Sciences, 
HBNI,
C.I.T Campus, Taramani, 
Chennai  600 113, 
India.}
\email{sanoli@imsc.res.in}

\address[W. Kohnen]
{Ruprecht-Karls-Universit\"at Heidelberg, 
Mathematisches Institut, 
Im Neuenheimer Feld 205, 
D-69120 Heidelberg, 
Germany.}
\email{winfried@mathi.uni-heidelberg.de}

\address[K. Soundararajan]   
{Stanford University, 
450 Jane Stanford Way, Building 380,
Stanford, CA 94305-2125.
USA.}
\email{ksound@stanford.edu}



\maketitle

\begin{abstract}   This article is concerned with the Fourier coefficients of cusp 
forms (not necessarily eigenforms) of half-integer weight lying 
in the plus space.  We give a soft proof that there are infinitely many 
fundamental discriminants $D$ such that the Fourier coefficients evaluated 
at $|D|$ are non-zero.  By adapting the resonance method, we also demonstrate 
that such Fourier coefficients must take quite large values.  
\end{abstract}

\section{\large \bf Introduction}

\smallskip

Let $k$ be a positive integer, and let $S_{k+\frac 12}$ denote the space of 
cusp forms of half-integral weight $k+\frac 12$ for the group $\Gamma_0(4)$ (which 
consists of the elements of $\Gamma = \SL_2({\Bbb Z})$ with lower left entry divisible by $4$). 
The theory of such forms was developed by Shimura \cite{GS}, and such a form $g$ has a 
Fourier expansion 
\begin{equation} 
\label{1.1} 
g(z) = \sum_{n=1}^{\infty} c(n) e^{2\pi inz}
\end{equation} 
with $z$ in the upper half plane.  We shall restrict attention to forms $g$ in the {\em plus subspace} 
$S_{k+\frac 12}^+$ of those forms whose Fourier coefficients $c(n)$ are zero unless $(-1)^k n \equiv 
0, 1  \pmod 4$ (see Kohnen \cite{WK}).  This paper is concerned with the coefficients $c(|D|)$ where 
$D$ is a fundamental discriminant with $|D|=(-1)^k D >0$.  In particular, we wish to show that these 
coefficients must be non-zero infinitely often, and indeed occasionally get large in terms of $|D|$. 
 For the sake of simplicity, we have restricted attention to level $4$ and to holomorphic forms 
 of half-integer weight, and it should be possible to extend these results to general level, or to non-holomorphic 
 Maass forms.

\smallskip

When $g$ is a Hecke eigenform,  Waldspurger's famous theorem (see \cite{JW}, 
and in fully explicit form \cite{KZ}) states that  
the squares $|c(|D|)|^2$ are proportional to the values $|D|^{k-{1\over 2}}L(f,\chi_D,k)$,
where $f$ is a normalized Hecke eigenform in the space $S_{2k}$ of cusp forms of 
weight $2k$ on $\Gamma$ corresponding to $g$ under the 
Shimura correspondence. Here, $L(f,\chi_D,s)$ denotes the Hecke 
$L$-function of $f$ twisted with the primitive quadratic character $\chi_D$ 
attached to the fundamental discriminant $D$,  and $s=k$ is its central point. In this case, the 
problems of non-vanishing and producing large values of $|c(|D|)|$ amount to the well studied 
problems of non-vanishing and omega results for central values in this family of 
$L$-functions (see, for example, \cite{HL, I, PP, RS, SY}).   
More recently, Hulse et al. \cite{HKKL} have studied sign changes in the coefficients $c(|D|)$'s 
(when normalized to be all real), and further progress on that problem
 is due to Lester and Radziwi{\l \l} \cite{LR}.

\smallskip 

 Our main interest, however, is 
in the situation where $g$ is a general cusp form, and not necessarily a Hecke eigenform.   In the 
case when $g$ is a linear combination of two eigenforms, the problem of non-vanishing was resolved by 
Luo and Ramakrishnan \cite{LR2}.   The general case was resolved in the work of Saha \cite{AS} who showed that 
  for {\it any} non-zero $g$ in $S_{k+{1\over 2}}^+$ there
are infinitely many fundamental discriminants $D$ with $(-1)^kD>0$ 
such that $c(|D|)$ is not zero.  In this paper we give two proofs of this result, showing further that $|c(|D|)|$ 
gets large for many fundamental discriminants $D$.  Our first proof introduces a new Dirichlet series built out of 
the coefficients $c(|D|)$ and Dirichlet $L$-functions attached to the character $\chi_D$.   This proof is qualitative and 
soft, and makes no use of the Waldspurger formula. 
The second proof is based on Waldspurger's formula and the connection to $L$-values.  
It uses the resonance method,  
developed in \cite{KS}, to show that linear combinations of $L$-values can be made large.   The resonance method proceeds by 
comparing the average of $L$-values weighted by a carefully chosen resonator Dirichlet polynomial with the average of the resonator polynomial itself.  If the ratio of these averages can be made large, then one concludes that the $L$-values must get large.  The new feature in our work is to show that a resonator that makes the twists of one $L$-function large does not correlate with twists of other $L$-functions, allowing one to obtain large values of linear combinations of $L$-functions.

\smallskip

\begin{thm} \label{thm1}  Let $g$ be a non-zero element of $S_{k+\frac 12}^+$ 
with Fourier expansion as in \eqref{1.1}.   

(a) There are infinitely many fundamental discriminants $D$ with $(-1)^k D>0$ such that $c(|D|) \neq 0$.  

(b)  Let $\epsilon >0$ be given, and $X$ be large.  There are at least $X^{1-\epsilon}$ 
fundamental discriminants $D$ with $X < (-1)^k D \le 2X$  such that 
$$ 
|c(|D|)| \ge  |D|^{\frac k2 -\frac 14} \exp\Big( \frac{1}{82} \frac{\sqrt{\log |D|}}{\sqrt{\log \log |D|}}\Big).
$$
\end{thm} 

\smallskip

If $g$ is an eigenform, then as mentioned earlier $|c(|D|)|^2$ is proportional 
to  $|D|^{k-\frac 12}L(f, \chi_D, k)$.   The Lindel{\" o}f hypothesis then implies that 
$|c(|D|)| \ll |D|^{\frac k2 - \frac 14+ \epsilon}$ for any $\epsilon >0$.   Writing 
a general $g \in S_{k+\frac 12}^+$ as a linear combination of eigenforms, 
we arrive at the conjecture that 
\begin{equation*}\label{1.2}
|c(|D|)| \ll_{g,\varepsilon}|D|^{{k\over 2}-{1\over 4}+\varepsilon} \qquad (\varepsilon >0).
\end{equation*}
This is an analogue of the Ramanujan-Petersson conjecture in integral weight, 
and remains an outstanding open 
problem.  Indeed,  conjectures on the maximal size of $L$-functions (see \cite{FGH}) 
suggest that for fundamental discriminants $D$, perhaps even the following 
stronger bound holds (for some $C>0$): 
$$ 
|c(|D|)| \ll |D|^{\frac k2 -\frac 14} \exp\Big( C \sqrt{\log |D| \log \log |D|}\Big). 
$$ 

\smallskip

The resonance method \cite{KS} produces large values of $L$-functions in very general settings. 
Although this is not one of the examples worked out in \cite{KS}, the resonance method shows 
that for Hecke eigenforms $f$ of integer weight $2k$, there are infinitely many fundamental 
discriminants $D$ such that  $L(f, \chi_D,k) \gg \exp(c\sqrt{\log |D|/\log \log |D|})$ for a positive 
constant $c$; a somewhat weaker result may be found in \cite{HL}.  Thus for an eigenform 
$g\in S_{k+\frac 12}^+$, one would get corresponding lower bounds for $|c(|D|)|$.  
Theorem \ref{thm1}(b) establishes a similar bound for general $g \in S_{k+\frac 12}^+$, 
and the key is to adapt the resonance method to show that one can produce large values 
of twists of a particular $L$-function  while keeping the twists of all other $L$-functions of 
Hecke eigenforms of weight $2k$ small.  
Work of Bondarenko and Seip \cite{BS} gives an improvement of the resonance method 
of \cite{KS}, producing still larger values of $|\zeta(\tfrac 12+it)|$, and a similar improvement 
for values of $|L(\tfrac 12, \chi)|$ has been obtained in \cite{BT}. However, this method 
exploits positivity (of coefficients, and of orthogonality relations) in crucial ways, and 
does not seem to extend to $L$-functions in other families, such as the family of 
quadratic twists of an eigenform.   Thus, apart from the constant $\frac 1{82}$ 
(which we have made no attempt to optimize), the lower bounds furnished in 
Theorem \ref{thm1} (b) are the best currently known, even in the situation of eigenforms $g$.  

\smallskip 

While Theorem \ref{thm1} (b) produces occasional large values of $|c(|D|)|$, ``typical" values 
of $|c(|D|)|$ tend to be much smaller.  Central values of $L$-functions are conjectured to be 
log-normal with a suitable mean and variance, which is a conjectured analogue of the 
classical work of Selberg on the log-normality of $|\zeta(\tfrac 12+it)|$.   Radziwi{\l \l} 
and Soundararajan \cite{RS} have established one sided central limit theorems for central values 
of quadratic twists 
of elliptic curves with positive sign of the functional equation.  These arguments carry over 
to quadratic twists of eigenforms of larger integer weight, and establish that for all but $o(X)$ 
fundamental discriminants $D$ with $X \le (-1)^k D \le 2X$ one has 
$$
|c(|D|)| \ll_{g, \epsilon} {|D|^{\frac k2 -\frac 14} (\log D)^{-\frac 14 +\epsilon}},
$$ 
where $\epsilon >0$.
The connection with $L$-functions first establishes such a result for eigenforms $g$, and 
then the same conclusion holds for any $g\in S_{k+\frac 12}^+$ by decomposing $g$
in terms of eigenforms.  

\smallskip

In our discussion above, we have confined ourselves to $c(|D|)$ where $D$ is a fundamental discriminant.  
These are the fundamental objects of interest, and the problem of obtaining large values of $c(n)$ for 
$n$ not arising from fundamental discriminants is of a different flavor (and comparatively easier).  For example, 
fixing a fundamental discriminant $D_0$, and varying $m$, the Shimura lift implies that for eigenforms $g$ 
finding large values of $c(|D_0| m^2)$ amounts to finding large values of the 
Hecke eigenvalues $a(m)$ of the Shimura lift.  For work in this direction see \cite{GK}, \cite{SD}.  
    
\medskip 

{\bf Acknowledgments.} S.G. would like to acknowledge MTR/2018/000201 and DAE number theory plan project
for partial financial support. K.S. is partially supported through a grant from the National Science Foundation, 
and a Simons Investigator Grant from the Simons Foundation. This work was carried out while K.S. 
was a senior Fellow at the ETH Institute for Theoretical Studies, whom he thanks for 
their warm and generous hospitality.

\section{\large \bf Notation and review}

\subsection{Half integer weight forms.}  Throughout let $g_1$, $\ldots$, $g_r$ denote a basis of Hecke eigenforms 
for $S_{k+\frac 12}^+$.  Denote the Fourier expansions of $g_{\nu}$ by 
$$ 
g_{\nu}(z) =\sum_{n=1}^{\infty} c_{\nu}(n) e^{2\pi inz}. 
$$ 
Let $g$ be a general cusp form in the space $S_{k+\frac 12}^+$, and write 
$$ 
g= \sum_{\nu=1}^{r} \lambda_{\nu} g_\nu 
$$ 
for some constants $\lambda_{\nu} \in  {\Bbb C}$.   Thus the Fourier coefficients $c(n)$ of $g$ 
are also linear combinations of the Fourier coefficients $c_\nu(n)$: 
$$ 
c(n) = \sum_{\nu=1}^{r} \lambda_\nu c_{\nu}(n). 
$$ 
The Fourier coefficients $c(n)$ satisfy the usual Hecke bound 
$$ 
c(n) \ll_g n^{\frac k2 + \frac 14},
$$ 
while they are expected to satisfy the analogue of the Ramanujan bound namely $c(n) \ll_g n^{\frac k2 -\frac 14+\epsilon}$ 
(which we discussed earlier in the case when $n=(-1)^kD$ for a fundamental discriminant $D$).

\smallskip

We associate to $g$ the Hecke $L$-series
$$
L(g,s)=\sum_{n\geq 1}c(n)n^{-s},  
$$
which by the Hecke bound for $c(n)$ converges absolutely when $\sigma >\frac k2 +\frac 54$.  
Further from \cite{GS} we know that 
 $L(g,s)$ has holomorphic continuation to ${\C}$ and satisfies the functional equation
\begin{equation} 
\label{2.1} 
\Lambda (g|W_4, ~k+ \tfrac{1}{2}-s)=\Lambda (g,s).
\end{equation} 
Here 
$$
g\mapsto g|W_4, ~~~(g|W_4)(z):=(-2iz)^{-k-{1\over 2}}g\Big(-{1\over {4z}}\Big)
$$
is the Fricke involution on $S_{k+{1\over 2}}$, and 
$$
\Lambda(g,s):=\pi^{-s}\Gamma(s)L(g,s). 
$$

\subsection{Dirichlet $L$-functions}   Associated to a fundamental discriminant $D$ 
is a primitive Dirichlet character $\pmod{|D|}$ which we denote by $\chi_D$.   
To $\chi_D$ we may associate the Dirichlet $L$-function 
$$ 
L(\chi_D,s) = \sum_{n=1}^{\infty} \chi_D(n)n^{-s}, 
$$
which converges absolutely in the half-plane $\sigma >1$ and extends analytically to ${\Bbb C}$
(except for a pole at $s=1$ in the case $D=1$ corresponding to $\zeta(s)$).  Put 
$\delta= 0$ if $D>0$ (so that $\chi_D(-1)=1$) and $\delta =1$ if $D<0$ (so that $\chi_D(-1)=-1$).  
 Then the completed $L$-function 
$$
\Lambda(\chi_D,s) :=~
\Big({|D| \over \pi}\Big)^{\frac{s + \delta}2}~\Gamma \Big({{s+\delta}\over 2}\Big)~L(\chi_D,s)
$$
satisfies the functional equation 
\begin{equation}\label{2.2}
\Lambda(\chi_D, 1-s) 
~=~  \Lambda(\chi_D, s).  
\end{equation}
For these classical facts, see for example Iwaniec and Kowalski \cite{IK}.  
\smallskip

\subsection{The Shimura lift and integer weight Hecke eigenforms}  The Shimura correspondence associates to 
every eigenform $g_\nu \in S_{k+\frac 12}^+$ a Hecke eigenform $f_{\nu}$ of weight $2k$ for the full modular group $\SL_2(\Z)$ (see \cite{WK}).  
We normalize $f_\nu$ to have first coefficient $1$, so that it has a Fourier expansion
$$ 
f_{\nu}(z) = \sum_{n=1}^{\infty} a_{\nu} (n) e^{2\pi i nz}, 
$$ 
with $a_{\nu}(1) = 1$.   The Fourier coefficients $a_{\nu}(n)$, which are also the eigenvalues of the Hecke operators, satisfy multiplicative 
Hecke relations, and satisfy the Deligne bound $|a_{\nu}(n)| \le d(n) n^{k-\frac 12}$ with $d(n)$ denoting the number of divisors of $n$.  
Associated to the Hecke eigenform $f_{\nu}$ is the $L$-function 
$$ 
L(f_\nu, s) = \sum_{n=1}^{\infty} \frac{a_{\nu}(n)}{n^s} = \prod_{p} \Big( 1- \frac{a_{\nu}(p)}{p^s} + \frac{p^{2k-1}}{p^{2s}}\Big)^{-1},  
$$
which converges absolutely for $\sigma > k+\frac 12$, extends analytically to ${\Bbb C}$, and satisfies the functional equation 
\begin{equation} 
\label{2.3}
\Lambda(f_{\nu},s) = (2\pi)^{-s} \Gamma(s) L(f_{\nu},s) = (-1)^k \Lambda(f_{\nu}, 2k-s). 
\end{equation}

The coefficients of $g_\nu$ and its Shimura lift $f_\nu$ are related by means of the identity 
$$ 
c_\nu(n^2 |D|) = c_\nu(|D|) \sum_{d|n} \mu(d) \chi_D(d) d^{k-1} a_\nu(n/d),
$$ 
where $D$ is a fundamental discriminant with $(-1)^k D>0$ and $n\geq 1$
or equivalently by means of the Dirichlet series identity 
\begin{equation} 
\label{2.4} 
L(\chi_D, s-k+1) \sum_{n=1}^{\infty} c_{\nu}(|D| n^2) n^{-s} = c_\nu(|D|) L(f_\nu, s). 
\end{equation} 

A deeper relation between the coefficients of $g_\nu$ and the Shimura lift $f_\nu$ is 
given by the Waldspurger formula.  If $D$ is a fundamental discriminant, the $L$-series 
of the $D$-th quadratic twist of $f_\nu$ is given by 
$$ 
L(f_\nu, \chi_D, s) = \sum_{n=1}^{\infty} a_\nu(n) \chi_D(n) n^{-s}. 
$$ 
It converges absolutely for $\sigma > k+\frac 12$, extends analytically to ${\Bbb C}$, and 
satisfies the functional equation 
\begin{equation} 
\label{2.5} 
\Lambda(f_\nu, \chi_D, s) = \Big( \frac{|D|}{2\pi}\Big)^s \Gamma(s) 
L(f_\nu, \chi_D,s) = (-1)^k \chi_D(-1) \Lambda(f_\nu, \chi_D, 2k-s).
 \end{equation} 
 Note that if $D$ is a fundamental discriminant with $(-1)^k D<0$, then the sign 
 of the functional equation above is 
$-1$, and so the central value $L(f_\nu, \chi_D, k)$ equals zero.  In the 
complementary case $(-1)^k D>0$ (which dovetails with 
the definition of the plus space $S_{k+\frac 12}^+$), Waldspurger's formula gives 
\begin{equation} 
\label{2.51} 
|c_\nu(|D|)|^2 = C_\nu |D|^{k-\frac 12} L(f, \chi_D,k).  
\end{equation} 
Here $C_\nu$ is a constant, which Kohnen and Zagier \cite{KZ} obtained in the elegant form 
\begin{equation} 
\label{2.52} 
C_\nu = \frac{(k-1)!}{\pi^{k}} \frac{\langle g_\nu, g_\nu \rangle}{\langle f_\nu, f_\nu \rangle},
\end{equation} 
where $\langle g_\nu, g_\nu\rangle$ and $\langle f_\nu, f_\nu \rangle$ denote 
the normalized Petersson norms of $g_\nu$ and $f_\nu$.

 Finally, we record a consequence of  Rankin--Selberg theory for the coefficients $a_{\nu}(p)$.  Namely, as 
 $x\to \infty$ 
\begin{equation} 
\label{2.6} 
\sum_{p\le x} \frac{|a_{\nu}(p)|^2}{p^{2k-1}} \log p \sim x, 
\end{equation} 
whereas if $\nu_1 \neq \nu_2$ then 
\begin{equation} 
\label{2.7} 
\sum_{p\le x} \frac{a_{\nu_1}(p) \overline{a_{\nu_2}(p)}}{p^{2k-1}} \log p = o(x). 
\end{equation}

\section {\large \bf Non-vanishing of Fourier coefficients}  

 \medskip
 
In this section we establish part (a) of Theorem \ref{thm1}, and  show that if $g \in S_{k+\frac 12}^+$ is not identically zero, then there are infinitely many fundamental  discriminants $D$ with $|D|= (-1)^kD>0$ such that $c(|D|) \neq 0$.  
Our proof will be based on 
 the following Dirichlet series: 
 \begin{equation} 
 \label{3.1} 
 D_g(s): = \sum_{n=1}^{\infty} \frac{\alpha(n)}{n^s}, 
 \end{equation}  
 where, writing $n$ uniquely as $n=|D|m^2$ with $D$ a fundamental discriminant as above, 
 $$
 \alpha(n):= c(|D|)\mu(m)\chi_D(m)m^{k-1}.
$$
The Hecke bound $|c(|D|)| \ll_g |D|^{\frac k2+\frac 14}$ gives $|\alpha(n)| \ll_g n^{\frac k2 +\frac 14}$ 
so that the Dirichlet series $D_g(s)$ converges absolutely in $\sigma >\frac k2 +\frac 54$, and defines a holomorphic function 
in that half-plane.  

In this half-plane of absolute convergence $\sigma > \frac k2 +\frac 54$, we may rewrite $D_g(s)$ as 
\begin{equation} 
\label{3.2} 
D_g(s) = \sum_{(-1)^k D >0} \sum_{m=1}^{\infty} \frac{c(|D|) \chi_D(m) \mu(m)}{|D|^s m^{2s-k+1}} 
= \sum_{(-1)^k D>0} \frac{c(|D|)}{|D|^s L(\chi_D, 2s-k+1)},
\end{equation} 
upon recalling that in the half-plane Re$(z) >1$ one has  
$$ 
\frac{1}{L(\chi_D, z)} = \sum_{m=1}^{\infty} \frac{\mu(m)\chi_D(m)}{m^z}. 
$$

 Since $g= \sum_{\nu} \lambda_\nu g_\nu$ we have 
 $$ 
 D_g(s) = \sum_{\nu=1}^{r} \lambda_\nu D_{g_\nu}(s). 
 $$ 
 Now from \eqref{2.4} (taking there $2s$ in place of $s$) we have 
 $$ 
\frac{ c_{\nu}(|D|)}{|D|^s L(\chi_D, 2s -k +1) } = \frac{1}{L(f_{\nu}, 2s)} \sum_{m=1}^{\infty} \frac{c_{\nu}(|D|m^2)}{|D|^s m^{2s}}, 
$$ 
and summing this over all $D$ with $(-1)^k D >0$ we conclude that 
$$ 
D_{g_{\nu}}(s) = \frac{L(g_\nu, s)}{L(f_{\nu}, 2s)}. 
$$ 
Thus 
 \begin{equation}\label{3.3}
D_g(s)=\sum_{\nu=1}^r\lambda_\nu{{L(g_\nu,s)}\over {L(f_\nu,2s)}}.
\end{equation}
 In particular, $D_g(s)$ 
has meromorphic continuation to $\C$ and is holomorphic for 
$\sigma > {k\over 2}+{1\over 4}$ (since in that half plane $L(f_{\nu},2s)$ has 
an absolutely convergent Euler product, and is therefore non-zero).  
 
\medskip
 
 Suppose now that $g\in S_{k+\frac 12}^+$   
 has only finitely many fundamental discriminants $D$ with $c(|D|) \neq 0$.  
 We seek to show that $g$ must be identically zero; that is, all the $\lambda_\nu$ 
 equal zero.  The proof is in two stages: First, we show that $D_g(s)$ must 
 be identically zero (that is, all the coefficients $c(|D|)$ are zero).  The key input here is 
 that if only finitely many $c(|D|)$ are non-zero, then from \eqref{3.2} $D_g(s)$ inherits a 
 functional equation arising from the one for Dirichlet $L$-functions.  But this turns out to be 
 inconsistent with the functional equation for $D_g(s)$ arising from \eqref{3.3} and the functional equations 
 for $L(g_\nu,s)$ and $L(f_\nu, 2s)$.  
 In the second  stage,  using these functional 
 equations again, we show that $\sum_{\nu =1}^{r} \lambda_\nu c_\nu(|D|) a_\nu(p)$ 
 must vanish for all fundamental discriminants $D$ with $4|D$ and $(-1)^k D>0$ and all odd primes $p$.  
 By invoking Rankin-Selberg relations for $a_\nu(p)$ together with the fact that for each eigenform $g_\nu$ 
 there exists a fundamental discriminant $D$ with $4|D$ and $c_\nu(|D|) \neq 0$ (see \cite{WK}), we finally 
 find that $g=0$; a contradiction.  
  
\medskip

\subsection{Showing that $D_g(s)=0$.}  From the functional equations 
\eqref{2.1} and \eqref{2.3} we see that 
$$ 
\frac{L(g_{\nu}, s)}{L(f_{\nu}, 2s)} = \gamma(s) \frac{L(g_{\nu} |W_4, k+\frac 12-s)}{L(f_{\nu},2k-2s)} 
$$ 
where, upon using the duplication formula for the $\Gamma$-function, 
$$ 
\gamma(s) = (-1)^k\cdot 2^{2k-4s}\cdot \pi^{k-{1\over 2}-2s}\cdot {{\Gamma(2s)
\Gamma(k+{1\over 2}-s})\over {\Gamma(s)\Gamma(2k-2s)}} = (-1)^k \pi^{k-\frac 12 -2s} \frac{\Gamma(s+\frac 12)}{\Gamma(k-s)}. 
$$
Thus we have the functional equation 
\begin{equation} 
\label{3.4} 
D_g(s) = \sum_{\nu=1}^r \lambda_\nu \frac{L(g_\nu,s)}{L(f_\nu,2s)} = 
\gamma(s) \sum_{\nu=1}^{r} \lambda_\nu \frac{L(g_\nu|W_4, k+\frac 12-s)}{L(f_\nu, 2k-2s)}. 
\end{equation}


\smallskip

On the other hand, if only finitely many $c(|D|)$ are non-zero, then 
we may use the functional equation for $L(\chi_D, s)$ (see \eqref{2.2}) in the expression 
\eqref{3.2}.  Thus, with $\delta=0$ if $k$ is even and $\delta =1$ if $k$ is odd,  
\begin{align} 
\label{3.5}
D_g(s) &= \sum_{(-1)^k D >0}{{c(|D|)}\over {|D|^sL(\chi_D,2s-k+1)}} \nonumber\\
&= \pi^{k- \frac{1}{2}-2s} \frac{\Gamma(\frac{2s-k+1+\delta}{2})}{\Gamma(\frac{k- 2s +\delta}{2})}  
\sum_{(-1)^k D>0} \frac{c(|D|}{|D|^{k-\frac 12 -s} L(\chi_D, k-2s)} \nonumber \\
&= \pi^{k- \frac{1}{2}-2s} \frac{\Gamma(\frac{2s-k+1+\delta}{2})}{\Gamma(\frac{k- 2s +\delta}{2})}   D_g(k- \tfrac{1}{2} - s). 
\end{align}
 We warn the reader that, unlike \eqref{3.4} which is a true functional equation, the relation 
 \eqref{3.5} is predicated on the assumption that only finitely many $c(|D|)$ are non-zero (which 
 we are attempting to disprove).  
 Combining this with \eqref{3.4}  (evaluated at $k-\frac 12-s$) we find that
\begin{equation}\label{eq4.12}
D_g(s)= R(s) \sum_{\nu=1}^r\lambda_\nu {{L(g_\nu|W_4, s + 1)}
\over {L(f_\nu, 2s + 1)}}, 
\end{equation}
where 
\begin{equation*}
R(s)= \pi^{k-\frac 12 -2s}  \frac{\Gamma(\frac{2s-k+1+\delta}{2})}{\Gamma(\frac{k- 2s +\delta}{2})}  \gamma(k-\tfrac 12-s) = 
 (-1)^k \cdot \frac{\Gamma(k-s)}{\Gamma(\frac{k+ \delta}{2} -s)}  \cdot
\frac{\Gamma(s + \frac{1+\delta -k}{2})}{\Gamma(1 + s)}.   
\end{equation*}
Since $k$ and $\delta$ have the same parity, $(k \pm \delta)/2$ is
always an integer, and so $R(s)$ is a rational function of $s$, being the ratio of two polynomials of degree
$(k - \delta)/2$. 

\smallskip

If $\sigma = \text{Re}(s)$ is large, then using the Hecke bound for Fourier coefficients of cusp forms we 
see that $L(g_\nu | W_4, s+1)$ is given by an absolutely convergent Dirichlet series.  Further, in such a 
half-plane, using the Euler product, $1/L(f_\nu, 2s+1)$ is also given by an absolutely 
convergent Dirichlet series.  
Thus,  we may view \eqref{eq4.12} as
\begin{equation}\label{n3}
D_g(s) = R(s) E_g(s),
\end{equation}
where $D_g(s)$ and $E_g(s)$ are both Dirichlet series in $s$, absolutely
convergent in some half plane. 

\smallskip 
We are now ready to establish our claim that $D_g(s)$ must be
identically zero. Suppose not, and consider the relation \eqref{n3}
for large real numbers $s$. For large real $s$, we have
\begin{equation*}
D_g(s) = a m^{-s} + O ( (m+1)^{-s} ),
\end{equation*}
where $a \ne 0$ and $m^{-s}$ is the first non-zero term in the Dirichlet
series for $D_g(s)$. Similarly
\begin{equation*}
E_g(s) = bn^{-s} + O ((n+1)^{-s}),
\end{equation*}
for large real $s$ (with $b \ne 0$), and so
$$
\frac{a}{m^s} \left(1 + O\left( \left(\frac{m}{m+1}\right)^s  \right) \right)
= R(s) \cdot \frac{b}{n^s} \left(1 + O\left( \left(\frac{n}{n+1}\right)^s  \right) \right).
$$
Since $R(s) \to (-1)^{k + \frac{k-\delta}{2}}$ as $s \to \infty$, clearly
we must have $m=n$ (and $a = (-1)^{k + \frac{k-\delta}{2}}b$).
But then we must have
$$
R(s) = (-1)^{k + \frac{k-\delta}{2}} + O\left( \left(\frac{n}{n+1}\right)^s  \right),
$$
which forces the rational function $R(s)$ to be a constant. 
Visibly this is a contradiction, and we conclude that 
$D_g(s)$ is identically zero.

\medskip

\subsection{Deducing that $g=0$}  The first stage of our proof has established that 
$c(|D|) =0$ for all fundamental discriminants $D$ with $(-1)^k D >0$.  It remains 
now to establish that $g$ is identically zero, or in other words $\lambda_\nu =0$ 
for all $\nu$.   

\smallskip
Since $D_g(s)$ is identically zero it follows from \eqref{eq4.12} that
 \begin{equation}\label{eq4.21}
\sum_{\nu=1}^r\lambda_\nu {{L(g_\nu|W_4, s+1)}
\over {L(f_\nu, 2s+1)}}=0
\end{equation}
for all $s$.   Precisely, we have the relation above for all $s$ not equalling a zero of the rational 
function $R(s)$, but there are only finitely many such $s$, and by analytic continuation the relation 
must hold for all $s$.  
 
Since $g_\nu$ is in the plus subspace we have
$$
g_\nu|U_4W_4= \Big({2\over{2k+1}}\Big)2^k g_\nu,
$$
where $U_4$ is the operator acting on power series by 
$\sum_{n\geq 0}c(n)q^n|U_4=\sum_{n\geq 0}c(4n)q^n$ (see \cite{WK} p. 250, and 
here $(\frac{2}{2k+1})$ denotes the Jacobi symbol). 
Since $W_4$ is an involution, applying $W_4$ to both sides of the above relation we find that 
$$ 
 g_\nu| U_4 = g_\nu|U_4 W_4 W_4 = \Big({2\over{2k+1}}\Big)2^k g_\nu| W_4. 
$$ 
 Thus, replacing also $s+1$ with $s$, we may rewrite \eqref{eq4.21} as 
\begin{equation}\label{eq4.22}
\sum_{\nu=1}^r\lambda_\nu {{L(g_\nu|U_4, s )}
\over {L(f_\nu, 2s-1)}}=0.
\end{equation}

Now recalling the definition of the $U_4$ operator, and the Euler product for $L(f_{\nu},2s-1)$ we 
see that for $\sigma$ sufficiently large 
$$ 
\frac{L(g_\nu| U_4, s)}{L(f_\nu, 2s-1)} = \Big( \sum_{n=1}^{\infty} \frac{c_{\nu}(4n)}{n^s} \Big) 
\prod_p \Big(1 - \frac{a_{\nu}(p)}{p^{2s -1}} + \frac{p^{2k-1}}{p^{4s-2}}\Big). 
$$ 
Now let $D$ be a fundamental discriminant with $4|D$ and with $(-1)^k D >0$, and let $p$ be an 
odd prime.   The coefficient 
of $(|D|p^2/4)^{-s}$ in the Dirichlet series above equals 
\begin{align*} 
c_{\nu}(|D|p^2) - c_{\nu}(|D|) \cdot pa_\nu(p) &= c_\nu(|D|) (a_\nu(p) - \chi_D(p) p^{k-1}) - c_{\nu}(|D|) \cdot pa_\nu(p) \\
&= c_{\nu}(|D|) \big( a_\nu(p) (1-p) - \chi_D(p) p^{k-1}\big),
\end{align*} 
where we used the  Shimura relation in the middle identity above.  From \eqref{eq4.22}, and since $c(|D|) =0$ 
as we have already established, we find 
$$ 
0 = \sum_{\nu=1}^{r} \lambda_\nu  c_{\nu}(|D|) \big( a_\nu(p) (1-p) - \chi_D(p) p^{k-1}\big) = (1-p) \sum_{\nu=1}^r \lambda_\nu c_\nu(|D|) a_{\nu}(p).  
$$ 
In other words, we conclude that for all fundamental discriminants $D$ with $4|D$ and $(-1)^k D>~0$, 
and all odd primes $p$ we have 
\begin{equation} 
\label{4.10} 
\sum_{\nu =1}^{r} \lambda_\nu c_\nu(|D|) a_{\nu}(p) = 0. 
\end{equation}

Recall that our goal is to show that all the $\lambda_\nu$ must be zero.  Suppose not, and 
(without loss of generality) that $\lambda_1 \neq 0$.  We know from (\cite{WK}, p. 260) 
that for each $\nu$ we can find a 
fundamental discriminant $D_\nu$ with $4|D_\nu$ and $(-1)^kD_\nu>0$, and 
such that $c_\nu(|D_\nu|)\neq 0$.  Apply \eqref{4.10} 
taking $D=D_1$ there, multiply the relation by $\overline{a_1(p)} p^{1-2k}\log p$, 
and sum over all $3 \le p\le x$.  Then 
$$ 
0 = \sum_{\nu=1}^r \lambda_\nu c_\nu(|D_1|) \sum_{3 \le p\le x} a_\nu(p) \frac{\overline{a_1(p)}}{p^{2k-1}} \log p 
\sim \lambda_1 c_1(|D_1|) x,
$$ 
by the Rankin--Selberg estimates \eqref{2.6} and \eqref{2.7}.  This contradiction completes our proof.

 \section{\large \bf Large values of Fourier coefficients}

\smallskip In this section, we begin our proof of part (b) of Theorem \ref{thm1}.  Using Waldspurger's formula, we recast the problem in terms of producing large values of a particular $L$-function while keeping other $L$-values small; see Theorem \ref{thm2} below.  We then show how to deduce Theorem \ref{thm2} from two technical propositions, which will be established in the following sections.  

Let $g =\sum_\nu \lambda_\nu g_\nu$ be a non-zero element in $S_{k+\frac 12}^+$.  Assume without loss of 
generality that $\lambda_1=1$, and that $|\lambda_\nu| \le 1$ for all $\nu =2, \ldots, r$.  By the triangle inequality and Cauchy-Schwarz, we 
obtain  
\begin{equation*} 
|c(|D|)| \ge |c_1(|D|)| - \sum_{\nu=2}^r |c_\nu(|D|)| 
\ge |c_1(|D|)| - \sqrt{r-1} \Big( \sum_{\nu=2}^{r} |c_\nu(|D|)|^2\Big)^{\frac 12}.  
\end{equation*} 
Applying  Waldspurger's formula \eqref{2.51}, it follows that 
\begin{equation} 
\label{5.1} 
|c(|D|)| \ge \big( C_1 |D|^{k-\frac 12} L(f_1, \chi_D,k)\big)^{\frac 12} - 
C \Big( |D|^{k-\frac 12} \sum_{\nu=2}^{r} L(f_\nu, \chi_D, k) \Big)^{\frac 12}, 
\end{equation} 
where $C_1$ is as in \eqref{2.52} (with $\nu=1$ there), and $C>0$ is a 
constant (depending on $f_\nu$, $g_\nu$, but independent of $D$).  
Theorem \ref{thm1} may now be deduced from the following result, which exhibits 
large values of   $L(f_1, \chi_D,k)$ while controlling $L(f_\nu, \chi_D,k)$  for $\nu=2, \ldots, r$.  

\begin{thm}  \label{thm2} Let $A >0$ be a constant, and let $X$ be large.  For 
any $\epsilon >0$, there are $\gg X^{1-\epsilon}$ fundamental 
discriminants $D$ with $X < (-1)^kD \le 2X$ such that 
$$ 
L(f_1, \chi_D,k) > A \sum_{\nu =2}^{r} L(f_\nu, \chi_D, k) + 
\exp\Big( \frac{1}{40} \frac{\sqrt{\log X}}{\sqrt{\log \log X}}\Big). 
$$ 
\end{thm}

To establish Theorem \ref{thm2} we shall use the resonance method.  
Let $D$ be fundamental discriminant with $X < (-1)^k D \le 2X$.  
We consider the following special value of a ``resonator" Dirichlet polynomial 
at $k - \frac{1}{2}$ :
\begin{equation} 
\label{5.2} 
R(D) = \sum_{n\le N} r(n) \frac{a_1(n)}{n^{k-\frac 12}} \chi_D(n), 
\end{equation} 
where $N=X^{\frac{1}{24}}$ and $r(n)$ is a multiplicative function defined as follows.  
Set $r(n)=0$ unless $n$ is square-free, and for primes $p$ define, 
with $L=\frac 18\sqrt{\log N \log \log N}$ 
\begin{equation} 
\label{5.3} 
r(p) = 
\begin{cases} \frac{L}{\sqrt{p} \log p} &\text{ if  } L^2 \le p \le L^4\\ 
0 &\text{ otherwise}. 
\end{cases} 
\end{equation} 
The proof of Theorem \ref{thm2} will be based on the following two propositions. 

\begin{prop} 
\label{prop1} With notations as above, we have 
\begin{equation} 
\label{5.4} 
\sum_{\substack{ X < (-1)^k D \le 2X \\ D\equiv 1 \mod 4}} |R(D)|^2 \le \frac{2X}{\pi^2} {\mathcal R} + O(X) , 
\end{equation} 
where 
\begin{equation} 
\label{5.45} 
{\mathcal R} =  \prod_{L^2 \le p\le L^4} \Big(1+ r(p)^2 \frac{a_1(p)^2}{p^{2k-1}} \Big).
\end{equation} 
Further 
\begin{equation} 
\label{5.5} 
\sum_{\substack{ X < (-1)^k D \le 2X \\ D\equiv 1\mod 4}} |R(D)|^6 \ll X \exp\Big( O\Big( \frac{\log X}{\log \log X}\Big)\Big). 
\end{equation} 
\end{prop} 
 
 \begin{prop} 
 \label{prop2}  With notations as above, we have 
 \begin{equation} 
 \label{5.6} 
 \sum_{\substack{ X < (-1)^k D \le 2X \\ D\equiv 1\mod 4}} L(f_1, \chi_D, k) |R(D)|^2 
 \gg X {\mathcal R} \exp\Big( (1+o(1)) \frac{L}{\log L}\Big), 
 \end{equation} 
 while for all $2\le \nu \le r$ 
 \begin{equation} 
 \label{5.7}
 \sum_{\substack { X < (-1)^k D \le 2X \\ D\equiv 1 \mod 4} } L(f_\nu, \chi_D, k) |R(D)|^2 
 \ll X{\mathcal R} \exp\Big( o\Big(\frac{L}{\log L}\Big)\Big). 
 \end{equation}  
 \end{prop}

We postpone the proof of these propositions to the next two sections, showing now how to 
deduce Theorem \ref{thm2} from them.  

\begin{proof}[Proof of Theorem \ref{thm2}]  Let ${\mathcal S}$ denote the set of 
fundamental discriminants $D$ with 
$X < (-1)^k D \le 2X$ and $D\equiv 1\mod 4$ with 
$$ 
L(f_1,\chi_D, k)  > A \sum_{\nu=2}^r L(f_\nu, \chi_D, k) 
+ \exp\Big( \frac 1{40} \frac{\sqrt{\log X}}{\sqrt{\log \log X}}\Big). 
$$ 
Note that 
\begin{align} 
\label{5.8} 
\sum_{\substack{X < (-1)^k D \le 2X \\ D\equiv 1\mod 4}} L(f_1, \chi_D,k) |R(D)|^2 &\le 
\sum_{\substack{X < (-1)^k D \le 2X \\ D\equiv 1\mod 4}} 
\Big(A \sum_{\nu=2}^r L(f_\nu, \chi_D,k) + \exp\Big( \frac 1{40} 
\frac{\sqrt{\log X}}{\sqrt{\log \log X}}\Big)\Big) |R(D)|^2 \nonumber \\ 
& \hskip .5 in + \sum_{D\in {\mathcal S}} L(f_1, \chi_D,k) |R(D)|^2 .
\end{align} 
Now \eqref{5.6} gives a lower bound for the left side above,
\begin{align*}
 \sum_{\substack{X < (-1)^k D \le 2X \\ D\equiv 1\mod 4}}
  L(f_1, \chi_D,k) |R(D)|^2 &\gg X {\mathcal R} 
 \exp\Big( \Big(\frac1{2} +o(1) \Big) \frac{L}{\log L} \Big) \\
 &= X{\mathcal R} \exp\Big( \Big(\frac{1}{8\sqrt{24}}+o(1)\Big) 
 \frac{\sqrt{\log X}}{\sqrt{\log \log X}}\Big),
 \end{align*}
 while \eqref{5.7} and \eqref{5.4} show that the first sum on the 
right side of \eqref{5.8} is negligible in comparison.  Thus we may conclude that 
\begin{equation} 
\label{5.9} 
 \sum_{D\in {\mathcal S}} L(f_1, \chi_D, k) |R(D)|^2  
 \gg X {\mathcal R} \exp\Big( \frac{1}{40} \frac{\sqrt{\log X}}{\sqrt{\log \log X}}\Big). 
\end{equation}  
 Applications of the Cauchy-Schwarz and H{\" o}lder inequalities show that the left side above is 
 \begin{align*}
 \le \Big( \sum_{\substack{ X< (-1)^k D \le 2X  \\ D\equiv 1\mod 4}} L(f_1, \chi_D, k)^2 \Big)^{\frac 12} 
 \Big( \sum_{D\in {\mathcal S}} |R(D)|^4 \Big)^{\frac 12} 
& \ll (X^{1+\epsilon})^{\frac 12} \Big( |{\mathcal S}|\Big)^{\frac{1}{6}} \Big( 
 \sum_{\substack{ X< (-1)^k D \le 2X \\ D\equiv 1 \mod 4}} |R(D)|^6\Big)^{\frac 13} 
\\
& \ll X^{\frac 56 +\epsilon} |{\mathcal S}|^{\frac 16}. 
 \end{align*} 
 Here we made use of \eqref{5.5} to bound the sum involving $|R(D)|^6$, and used the Perelli--Pomyka{\l}a bound \cite{PP}
 (obtained from Heath-Brown's large sieve for quadratic characters) of $X^{1+\epsilon}$ for the second moment of 
 $L$-values.  Theorem \ref{thm2} follows. 
\end{proof} 

We remark that the second moment of the central $L$-values should be of size $X \log X$, which would lead to a better quantification for the number of large values produced in Theorems \ref{thm1} and \ref{thm2}.   This second moment remains barely out of reach of present technology, but an asymptotic is known assuming GRH (see \cite{SY}).

\section{Proof of Proposition \ref{prop1}} 

\begin{lem} \label{lem1} Let $u\le X$ be an odd natural number.  If $u$ is a square then 
$$
\sum_{\substack{X < (-1)^k D \le 2X \\ D\equiv 1 \mod 4}} \chi_D(u)  = \frac{X}{2\zeta(2)} \prod_{p|2u} \Big(\frac{p}{p+1}\Big) + O(X^{\frac 12+\epsilon}u^{\frac 14}), 
$$ 
while if $u$ is not a square then 
$$ 
\sum_{\substack{X < (-1)^k D \le 2X \\ D\equiv 1 \mod 4} } \chi_D(u) \ll X^{\frac 12+\epsilon} u^{\frac 14}. 
$$
\end{lem}
\begin{proof} Let $\chi_0$ and $\chi_{-4}$ denote the principal and non-principal 
characters $\mod 4$.  Note that, for any non-zero integer $D$,  
$$ 
\frac 12 (\chi_0(D) + \chi_{-4}(D)) \sum_{\substack{ a^2|D \\ a\text{ odd} }} \mu(a) = 
\begin{cases} 
1 &\text{ if  } D \equiv 1\mod 4 \text{ is square-free}, \\
0&\text{ otherwise}. 
\end{cases}
$$ 
Thus writing $D= a^2 b$ with $b$ square-free, we see that 
 \begin{equation} 
 \label{6.1} 
\sum_{\substack{X < (-1)^k D \le 2X \\ D\equiv 1 \mod 4}} \chi_D(u)  = \frac 12 \sum_{\substack{a \le \sqrt{2X} \\ a \text{ odd}}} \mu(a) \sum_{X/a^2 < (-1)^k b \le 2X/a^2} (\chi_0(b) +\chi_{-4}(b)) \Big( \frac{a^2b}{u} \Big). 
\end{equation} 

If $u$ is not a square, then $\chi_0 (\cdot) (\frac{\cdot}{u})$ and $\chi_{-4}(\cdot) (\frac{\cdot}{u})$ are both non-principal Dirichlet characters to the modulus $4u$.  Therefore, using the P{\' o}lya--Vinogradov 
bound we obtain 
$$ 
 \sum_{X/a^2 < (-1)^k b \le 2X/a^2} (\chi_0(b) +\chi_{-4}(b)) \Big( \frac{a^2b}{u} \Big) \ll \min \Big( \sqrt{u} \log (4u), \frac{X}{a^2} \Big).
$$ 
Here the bound $\sqrt{u}\log (4u)$ comes from P{\' o}lya--Vinogradov, and the bound $X/a^2$ by estimating the sum over $b$ trivially.  Therefore 
in this case we obtain 
$$ 
\sum_{\substack{X < (-1)^k D \le 2X \\ D\equiv 1 \mod 4}} \chi_D(u)  \ll  \sum_{\substack{a \le \sqrt{2X} \\ a \text{ odd}}} \min \Big( \sqrt{u} \log (4u), \frac{X}{a^2} \Big) 
\ll  \sum_{\substack{a \le \sqrt{2X} \\ a \text{ odd}}} \Big( \sqrt{u} \log (4u) \frac{X}{a^2} \Big)^{\frac 12} \ll X^{\frac 12 +\epsilon} u^{\frac 14}. 
$$

If $u$ is a square, then $\chi_0(\cdot)(\frac{\cdot}{u})$ is a principal character, which contributes 
\begin{align*}
\frac 12 \sum_{\substack{ a\le \sqrt{2X} \\ (a,2u)=1}} \mu(a) \sum_{\substack{X/a^2 < (-1)^k b \le 2X/a^2 \\ (b,2u)=1}} 1
&= \frac 12 \sum_{\substack{a\le \sqrt{2X} \\ (a,2u)=1}} \mu(a) \Big( \frac{X}{a^2} \frac{\phi(2u)}{2u} + O(u^{\epsilon}) \Big) \\
&= \frac{X}{2\zeta(2)} \prod_{p|2u} \Big( \frac{p}{p+1} \Big) + O(X^{\frac 12 +\epsilon}).
\end{align*}
This completes the proof of the lemma.
\end{proof} 

We are now ready to prove Proposition \ref{prop1}.   Expanding out the definition of $R(D)$, we obtain 
\begin{align*} 
\sum_{\substack{X < (-1)^k D \le 2X \\ D\equiv 1 \mod 4}} |R(D)|^2 = \sum_{n_1, n_2 \le N} r(n_1) r(n_2) \frac{a_1(n_1)}{n_1^{k-\frac 12}} \frac{a_1(n_2)}{n_2^{k-\frac 12}} \sum_{\substack{X < (-1)^k D \le 2X \\ D\equiv 1 \mod 4}} 
\chi_D(n_1n_2),  
\end{align*} 
and we now use Lemma \ref{lem1} to estimate the sum over $D$.  Since $r(n)=0$ unless $n$ is odd and square-free, $|r(n)|\le 1$ always, and 
$|a_1(n)|/n^{k-\frac 12} \le d(n) \ll n^{\epsilon}$, we see that the error terms arising from Lemma \ref{lem1} contribute 
$$ 
\ll X^{\frac 12+\epsilon} \sum_{n_1, n_2\le N} (n_1 n_2)^{\frac 14+\epsilon} \ll X^{\frac 12+\epsilon} N^{\frac 52+\epsilon}.
$$ 
The main term in Lemma \ref{lem1} arises when $n_1n_2$ is a square, and since $n_1$ and $n_2$ are both square-free, this means that $n_1=n_2$.  
Thus the main term is 
$$ 
\frac{X}{2\zeta(2)} \sum_{n\le N} r(n)^2 \frac{a_1(n)^2}{n^{2k-1}} \prod_{p|2n} \Big( \frac{p}{p+1}\Big) \le \frac{2X}{\pi^2} \prod_{L^2 \le p \le L^4} \Big( 1+ r(p)^2 \frac{a_1(p)^2}{p^{2k-1}} \frac{p}{p+1}\Big) \le \frac{2 X}{\pi^2} {\mathcal R}, 
$$
upon extending the sum over $n$ to all natural numbers, and recalling the definition of the multiplicative function $r$.  This proves \eqref{5.4}.

The proof of \eqref{5.5} is similar.  We expand out $R(D)^6$ and use Lemma \ref{lem1}.  The error terms that arise are bounded now 
by
$$ 
\ll X^{\frac 12+\epsilon} \sum_{n_1, \ldots, n_6 \le N} (n_1\cdots n_6)^{\frac 14+\epsilon} \ll X^{\frac 12+\epsilon} N^{\frac{15}{2}+\epsilon}.
$$ 
The main term arises from terms with $n_1 \cdots n_6$ being a square, and for these terms $a_1(n_1)\cdots a_1(n_6)$ is always non-negative 
(since $n_i$ are all square-free, $a_1$ is a multiplicative function, and each prime dividing $n_1 \cdots n_6$ divides an even number of $n_i$).    
 Thus the main term is 
 $$ 
 =\frac{\pi^2 X}{9} \sum_{\substack{ n_1, \ldots, n_6\le N \\ n_1\cdots n_6 =\square}} r(n_1)\cdots r(n_6) \frac{a_1(n_1) \cdots a_1(n_6)}{(n_1\cdots n_6)^{k-\frac 12}} \prod_{p|n_1\cdots n_6} \Big(\frac{p}{p+1}\Big). 
 $$
 Extending the sum over $n_i$ to infinity, and using multiplicativity, the above is 
 \begin{align*}
& \ll X \prod_{L^2 \le p \le L^4} \Big( 1+ \binom{6}{2} r(p)^2 \frac{a(p)^2}{p^{2k-1}} + \binom{6}{4} r(p)^4 \frac{a(p)^4}{p^{4k-2}} + 
 \binom{6}{6} r(p)^6 \frac{a(p)^6}{p^{6k-3}} \Big)
 \\
 &\ll X \exp\Big( O\Big( \sum_{L^2 \le p \le L^4} \frac{L^2}{p (\log p)^2} \Big) \Big) \ll X \exp\Big( O\Big(\frac{\log X}{\log \log X}\Big)\Big). 
 \end{align*}
 This completes the proof of Proposition \ref{prop1}.

\section{Proof of Proposition \ref{prop2}} 

\begin{lem} \label{lem2}  Let $u$ be an odd positive integer, and write $u=u_1 u_2^2$ 
with $u_1$ square-free.  Let $\Phi$ denote 
a smooth function compactly supported in $[1/2,5/2]$, and with $0\le \Phi(t) \le 1$ for all $t$.   Then 
$$ 
\sum_{\substack{ (-1)^k D >0 \\ D\equiv 1\mod 4 }} \chi_D(u) L(f_\nu, \chi_D,k) \Phi\Big( \frac{|D|}{X}\Big) 
=   A_\nu h_\nu(u) \Big( \int_0^\infty \Phi(t) dt \Big) \frac{a_\nu(u_1)}{u_1^k} X 
+ O(X^{\frac 78 +\epsilon} u^{\frac 38}),  
$$ 
where $A_\nu$ is a non-zero constant, and $h_\nu$ is a multiplicative 
function with $h_\nu(p^t) = 1+ O(1/p^{t})$ 
for prime powers $p^t$.  
\end{lem} 
\begin{proof}   This is a variant of Proposition 2 of \cite{RS} which treats the case of 
quadratic twists of an elliptic curve.   Indeed Proposition 2 of \cite{RS} is a little 
more general in allowing the discriminants $D$ to lie in a given progression modulo 
the level, and also to restrict $D$ to be multiples of another parameter $v$.   Only 
minor modifications to that argument are needed to handle eigenforms of weight 
$k$ instead of elliptic curves.  The techniques involved are based on earlier work 
in the family of quadratic twists, see \cite{I, KS2, SY}.   Very briefly, we start with 
an ``approximate functional equation" 
$$ 
L(f_\nu, \chi_D, k) = 2 \sum_{n=1}^{\infty} \frac{a_\nu(n)}{n^k} \chi_D(n) W\Big( \frac{n}{|D|}\Big), 
$$ 
for a suitable weight function $W(\xi)$, which is approximately $1$ for small $\xi$ 
and decays rapidly as $\xi \to \infty$.  Then the sum we wish to evaluate equals 
$$
2\sum_{n=1}^{\infty} \frac{a_\nu(n)}{n^k} \sum_{\substack{ (-1)^k D >0 \\ D\equiv 1\mod 4 }} 
\chi_D(un) W\Big(\frac{n}{|D|}\Big) 
\Phi\Big(\frac{|D|}{X}\Big).
$$
The main terms arise from the case when $un$ is a perfect square, and the contribution 
of all other terms can be bounded as in \cite{RS}.  Since $u=u_1 u_2^2$, the condition 
$un$ being a square amounts to writing $n=u_1m^2$, and so the main term equals 
$$ 
2 \sum_{m=1}^{\infty} \frac{a_\nu(u_1 m^2)}{u_1^k m^{2k}} \sum_{\substack{  (-1)^k D >0 \\ D\equiv 1\mod 4  \\ (D,u_1u_2m)=1}}
W\Big( \frac{u_1m^2}{|D|} \Big) \Phi\Big( \frac{|D|}{X} \Big). 
$$ 
Evaluating the sum over $D$ asymptotically, we arrive at a main term 
$$ 
X \Big( \int_0^\infty \Phi(t) dt \Big) \frac{4}{\pi^2} \sum_{m=1}^{\infty} \frac{a_\nu(u_1 m^2)}{u_1^k m^{2k} } \prod_{\substack{ p|u_1 u_2 m \\ p>2 }} \Big( \frac{p}{p+1}\Big). 
$$ 
Using the Hecke relations, this can be put in the form stated in the lemma, and we note 
that the constant $A_\nu$ is closely related to the value of the symmetric square 
$L$-function attached to $f_\nu$ evaluated at the edge of the critical strip; see 
Proposition 2 of \cite{RS} for further details.  
\end{proof}

With this lemma in place, we are ready to evaluate 
$$ 
\sum_{\substack{ (-1)^k D> 0 \\ D\equiv 1 \mod 4}} L(f_\nu, \chi_D, k) |R(D)|^2 \Phi\Big( \frac{|D|}{X}\Big), 
$$ 
for $\nu =1$, $\ldots$, $r$, and $\Phi$ being a suitable approximation to the indicator function of $[1,2]$.  Expanding out $|R(D)|^2$ 
and using Lemma \ref{lem2} we see that the above equals 
\begin{align} 
\label{6.1} 
A_\nu \Big(\int_0^\infty \Phi(t) dt \Big)  X& \sum_{n_1, n_2 \le N} r(n_1) r(n_2) \frac{a_1(n_1) a_1(n_2)}{(n_1n_2)^{k-\frac 12}} h_\nu(n_1n_2) 
\frac{a_\nu (n_1 n_2/(n_1, n_2)^2)}{(n_1 n_2/(n_1, n_2)^2)^k} \nonumber \\
&+O\Big(X^{\frac 78+\epsilon} \sum_{n_1, n_2 \le N} r(n_1) r(n_2) 
\frac{|a_1(n_1) a_1(n_2)|}{(n_1n_2)^{k-\frac 12}}   (n_1 n_2)^{\frac 38} \Big). 
\end{align}
In deriving the above expression, we used that $n_1$ and $n_2$ are square-free (else $r(n_1) r(n_2)=0$) so that 
$n_1 n_2 = (n_1 n_2/(n_1, n_2)^2) (n_1, n_2)^2$ with $n_1 n_2/(n_1, n_2)^2$ being square-free.

Since $|a_1(n_1)| \le d(n_1) n_1^{k-\frac 12} \ll n_1^{k-\frac 12+\epsilon}$ by the Deligne bound, and $r(n_1) \le 1$ always, 
the error term in \eqref{6.1} 
may be bounded by 
$$
\ll X^{\frac 78+\epsilon} N^{\frac {11}4 +\epsilon} \ll X^{\frac{99}{100}}, 
$$ 
which is acceptable.  

We now analyze the main term in \eqref{6.1}.  First we extend the sums over 
$n_1$ and $n_2$ to all natural numbers and analyze this contribution, and then 
we show that the contribution of the terms 
with $\max (n_1, n_2) > N$ is negligible.  When the terms over $n_1$, $n_2$ 
are extended to all natural numbers, the resulting sums are multiplicative in nature, 
and thus these give 
\begin{equation} 
\label{6.2} 
A_\nu \Big(\int_0^\infty \Phi(t) dt \Big)  X \prod_{L^2 \le p \le L^4} \Big( 1 + 2 r(p) 
h_\nu(p) \frac{a_1(p) a_\nu(p) }{p^{2k-\frac 12}} 
+ r(p)^2 h_\nu (p^2) \frac{a_1(p)^2}{p^{2k-1}} \Big). 
\end{equation} 
In thinking of the Euler product above, the first term corresponds to $n_1$ and 
$n_2$ both not divisible by $p$, the middle term 
corresponds to exactly one of $n_1$ or $n_2$ being divisible by $p$, and the last term 
to both $n_1$ and $n_2$ being 
divisible by $p$.  

Now it remains to show that the terms with $\max(n_1,n_2) >N$ (which are not 
present in \eqref{6.1} but included in \eqref{6.2}) 
contribute a negligible amount.  These terms may be bounded by 
\begin{align*} 
&\ll X \sum_{\max (n_1, n_2) >N} r(n_1) r(n_2) \frac{|a_1(n_1) a_1(n_2)|}{(n_1n_2)^{k-\frac 12}} 
|h_\nu(n_1n_2)| 
\frac{|a_\nu(n_1n_2/(n_1,n_2)^2)|}{(n_1n_2/(n_1,n_2)^2)^k} \\
&\ll X \sum_{n_1, n_2 =1}^{\infty}  r(n_1) r(n_2) \frac{|a_1(n_1) a_1(n_2)|}{(n_1n_2)^{k-\frac 12}} 
|h_\nu(n_1n_2)| 
\frac{|a_\nu(n_1n_2/(n_1,n_2)^2)|}{(n_1n_2/(n_1,n_2)^2)^k} \Big( \frac{n_1 n_2}{N}\Big)^{\alpha}, 
\end{align*} 
for any $\alpha >0$.  By multiplicativity the above equals 
$$ 
X N^{-\alpha} \prod_{L^2 \le p\le L^4} 
\Big(1 + 2r(p)p^{\alpha}  |h_\nu(p)| \frac{|a_1(p)a_\nu(p)|}{p^{2k-\frac 12}} + r(p)^2 
p^{2\alpha}|h_\nu(p^2)| \frac{a_1(p)^2}{p^{2k-1}}\Big). 
$$ 
Since $h_\nu(p^t) = 1+O(1/p^t)$ and $|a_\nu(p)| \le 2 p^{k-\frac 12}$ this is 
$$ 
\ll X N^{-\alpha} \exp\Big( \sum_{L^2 \le p \le L^4} \Big( \frac{8Lp^{\alpha} }{p\log p} +
 \frac{4L^2p^{2\alpha}}{p (\log p)^2} \Big) \Big(1+ O\Big(\frac 1p\Big) \Big) \Big).  
$$
Upon choosing $\alpha= 1/(8\log L)$, and using the prime number theorem, the above is 
$$ 
\ll X\exp\Big(-\frac{\log N}{8\log L} +  \frac{8 L}{\log L} + \frac{2L^2}{(\log L)^2} \Big) \ll X, 
$$ 
recalling that $L = \frac 18 \sqrt{\log N \log \log N}$.  

From our work above we conclude that 
\begin{align} 
\label{6.3} 
\sum_{\substack{ (-1)^k D> 0 \\ D\equiv 1 \mod 4}} &L(f_\nu, \chi_D, k) |R(D)|^2 \Phi\Big( \frac{|D|}{X}\Big) \nonumber \\
&= A_\nu \Big(\int_0^\infty \Phi(t) dt \Big)  X \prod_{L^2 \le p \le L^4} \Big( 1 + 2 r(p) h_\nu(p) \frac{a_1(p) a_\nu(p) }{p^{2k-\frac 12}} 
+ r(p)^2 h_\nu (p^2) \frac{a_1(p)^2}{p^{2k-1}} \Big) +O(X). 
\end{align}
Let us compare the product above with the product ${\mathcal R}$.   For $L^2 \le p\le L^4$, note that 
(keeping in mind $r(p) = L/(\sqrt{p}\log p) \le 1/\log p$ is always small, that  
$h_\nu(p^t) = 1+O(1/p)$, and that $|a_\nu(p)|\le 2p^{k-\frac 12}$) 
\begin{align*} 
\Big( 1 + 2 r(p) h_\nu(p) \frac{a_1(p) a_\nu(p) }{p^{2k-\frac 12}} 
+ r(p)^2 h_\nu (p^2) \frac{a_1(p)^2}{p^{2k-1}} \Big) &\Big(1+ r(p)^2 \frac{a_1(p)^2}{p^{2k-1}}\Big)^{-1} \\
& = 1+ 2r(p) \frac{a_1(p)a_\nu(p)}{p^{2k-\frac 12}} + O\Big( \frac{r(p)^3}{\sqrt{p}}\Big) \\
&= \exp\Big( 2r(p) \frac{a_1(p)a_\nu(p)}{p^{2k-\frac 12}} + O\Big( \frac{r(p)^2}{\sqrt{p}}\Big) \Big). 
\end{align*}
Using the prime number theorem we conclude that the product in \eqref{6.3} equals 
\begin{align} 
\label{6.4} 
&{\mathcal R} \exp\Big( \sum_{L^2 \le p\le L^4} \Big( 2r(p) \frac{a_1(p)a_\nu(p)}{p^{2k-\frac 12}} 
+ O\Big( \frac{r(p)^2}{\sqrt{p}}\Big) \Big)\Big) 
\nonumber \\
= &{\mathcal R} \exp\Big( \sum_{L^2 \le p\le L^4}2r(p) \frac{a_1(p)a_\nu(p)}{p^{2k-\frac 12}} 
+ O\Big( \frac{L}{(\log L)^3} \Big)\Big). 
\end{align}

We are now ready to prove Proposition \ref{prop2}.  In the case $\nu =1$, take 
$1\ge \Phi(t)\ge 0$ to be a  smooth function 
supported on $[1,2]$ with $\Phi(t) = 1$ on $[1.1,1.9]$.  Then 
$$
\sum_{\substack{X <  (-1)^k D \le 2X \\ D\equiv 1 \mod 4}} L(f_1, \chi_D, k) |R(D)|^2 \ge 
\sum_{\substack{ (-1)^k D> 0 \\ D\equiv 1 \mod 4}} L(f_1, \chi_D, k) |R(D)|^2 \Phi\Big( \frac{|D|}{X}\Big), 
$$
and from \eqref{6.3} and \eqref{6.4}, we conclude that this is 
$$ 
\ge \frac{4}{5} A_1 X {\mathcal R} \exp\Big( \sum_{L^2 \le p\le L^4}2r(p) 
\frac{a_1(p)^2}{p^{2k-\frac 12}} + O\Big( \frac{L}{(\log L)^3} \Big)\Big) +O(X). 
$$
Applying the Rankin--Selberg estimate \eqref{2.6} and partial summation, we obtain 
$$ 
 \sum_{L^2 \le p\le L^4}2r(p) \frac{a_1(p)^2}{p^{2k-\frac 12}} 
 = 2 L\sum_{L^2 \le p\le L^4} \frac{a_1(p)^2}{p^{2k} \log p} 
 = \Big(\frac1{2}+o(1)\Big) \frac{L}{\log L}, 
 $$ 
 from which \eqref{5.6} follows. 
 
 Now we turn to the case $\nu >1$, where we take $1\ge \Phi(t) \ge 0$ to 
 be a smooth function compactly supported on $[1/2,5/2]$ and with  $\Phi(t)=1$ on $[1,2]$.      
 Now our work in \eqref{6.3} and \eqref{6.4} shows that 
 \begin{align*}
 \sum_{\substack{X <  (-1)^k D \le 2X \\ D\equiv 1 \mod 4}} L(f_\nu, \chi_D, k) |R(D)|^2 &\le 
\sum_{\substack{ (-1)^k D> 0 \\ D\equiv 1 \mod 4}} L(f_\nu, \chi_D, k) |R(D)|^2 \Phi\Big( \frac{|D|}{X}\Big) \\ 
&\le 2A_\nu X {\mathcal R} \exp\Big( \sum_{L^2 \le p\le L^4} 2 
r(p) \frac{a_1(p)a_\nu(p)}{p^{2k-\frac 12}} + O\Big(\frac{L}{(\log L)^3}\Big) \Big) +O(X). 
\end{align*}
Here the Rankin--Selberg estimate \eqref{2.7} and partial summation give 
$$ 
\sum_{L^2 \le p\le L^4}2r(p) \frac{a_\nu(p) a_1(p)}{p^{2k-\frac 12}} = o\Big( \frac{L}{\log L}\Big), 
$$ 
from which \eqref{5.7} follows.

\end{document}